\documentclass{amsart}
\usepackage{amssymb,latexsym}
\usepackage{amscd,amsthm}

\usepackage[all]{xy}
\usepackage{tikz}

\usepackage{tikz-cd}

\newtheorem{theorem}{Theorem}[section]

\newtheorem{lemma}[theorem]{Lemma}
\newtheorem{proposition}[theorem]{Proposition}
\newtheorem{corollary}[theorem]{Corollary}

\theoremstyle{definition}
\newtheorem{definition}[theorem]{Definition}

\newtheorem{example}[theorem]{Example}

\DeclareMathOperator{\Ext}{Ext}
\DeclareMathOperator{\Hom}{Hom}

\DeclareMathOperator{\im}{Im}

\DeclareMathOperator{\Ker}{Ker}

\newcommand{\cat}[1]{\mathcal{#1}}           %% font for categories

\newcommand{\class}[1]{\mathcal{#1}}   %% font for classes
\newcommand{\N}{\mathbb{N}}
\newcommand{\Z}{\mathbb{Z}}

\newcommand{\ch}{\textnormal{Ch}(R)}

\newcommand{\cha}[1]{\textnormal{Ch}(\mathcal{#1})}

\newcommand{\tilclass}[1]{\widetilde{\class{#1}}}
\newcommand{\dgclass}[1]{dg\widetilde{\class{#1}}}
\newcommand{\dwclass}[1]{dw\widetilde{\class{#1}}}
\newcommand{\exclass}[1]{ex\widetilde{\class{#1}}}

\newcommand{\rightperp}[1]{#1^{\perp}}
\newcommand{\leftperp}[1]{{}^\perp #1}

\begin{document}

\title[Homological dimensions over coherent regular rings]{Homological dimensions of complexes over coherent regular rings}

\author{James Gillespie}
\address{J.G. \ Ramapo College of New Jersey \\
         School of Theoretical and Applied Science \\
         505 Ramapo Valley Road \\
         Mahwah, NJ 07430\\ U.S.A.}
\email[Jim Gillespie]{jgillesp@ramapo.edu}
\urladdr{http://pages.ramapo.edu/~jgillesp/}

\author{Alina Iacob}
\address{A.I. \ Department of Mathematical Sciences \\
         Georgia Southern University \\
         Statesboro (GA) 30460-8093 \\ U.S.A.}
\email[Alina Iacob]{aiacob@GeorgiaSouthern.edu}

\date{\today}



\begin{abstract}%% what do you actually prove?%%
We show that Iacob--Iyengar's answer to a question of Avromov--Foxby extends from Noetherian to coherent rings. In particular, a coherent ring $R$ is regular if and only if the injective (resp. projective) dimension of each complex X of $R$-modules agrees with its graded-injective (resp. graded-projective) dimension.  The same is shown for the analogous dimensions based on FP-injective $R$-modules, and on flat $R$-modules. 
\end{abstract}

\maketitle

\section{Introduction}

Two different notions of homological dimensions for unbounded chain  complexes of modules over a ring were introduced by Avramov and Foxby in~\cite{avramov-foxby-hom-dims-complexes}.
In particular, given a ring $R$, let $\dgclass{P}$ denote the class of all \emph{DG-projective}  chain complexes of $R$-modules; they are the  complexes of projective modules that are K-projective in the sense of Spaltenstein~\cite{spaltenstein}. Then writing $X \simeq Y$ to mean there exists  a quasi-isomorphism between $X$ and $Y$, the  \emph{projective dimension} of a chain complex $X$ is defined:
\begin{equation}\label{eq-pd}\tag{$\dagger$}
\text{pd}_R(X) := \text{min}\{\,n \in \mathbb{Z} \,|\, X \simeq P  \text{ for some } P \in \dgclass{P} \text{ with } P_i = 0,  \ \forall \,i > n\,\}.
\end{equation}
If $X \simeq P\in\dgclass{P}$ implies $P$ is unbounded above, we set $\text{pd}_R(X) = \infty$.  At the other extreme, $X\simeq 0$ means $X$ is exact, and in this case we set $\text{pd}_R(X) = -\infty$.
If $X$ is a bounded below complex  or, in particular,  just an $R$-module concentrated in degree 0,  then this definition of $\text{pd}_R(X)$ agrees with the classical notion of the projective dimension of $X$ (\cite[Remark~2.8.P]{avramov-foxby-hom-dims-complexes}).

But one could instead replace  $\dgclass{P}$  with the class $\dwclass{P}$, of all \emph{degreewise projective} (i.e. \emph{graded-projective}) complexes. This defines the \emph{graded-projective dimension} of $X$:
\begin{equation}\label{eq-pd}\tag{$*$}
\text{gr-pd}_R(X) :=  \text{min}\{\,n \in \mathbb{Z} \,|\, X \simeq P  \text{ for some } P \in \dwclass{P} \text{ with } P_i = 0,  \ \forall \,i > n\,\}.\end{equation}
Since  $\dgclass{P}\subseteq\dwclass{P}$ is, in general, a strict containment we have the inequality $\text{gr-pd}_R(X)\leq \text{pd}_R(X)$.
After defining $\text{pd}_R(X)$ and $\text{gr-pd}_R(X)$, as well as their injective variants $\text{id}_R(X)$ and  $\text{gr-id}_R(X)$,  Avramov and Foxby showed that $\text{pd}_R(X) = \text{gr-pd}_R(X)$ and $\text{id}_R(X) = \text{gr-id}_R(X)$ whenever $R$ is a ring of finite global dimension (\cite[Corollary~3.5]{avramov-foxby-hom-dims-complexes}). They then asked (\cite[Question~3.8]{avramov-foxby-hom-dims-complexes}) if the converse holds: Does having the equalities $\text{pd}_R(X) = \text{gr-pd}_R(X)$, or $\text{id}_R(X) = \text{gr-id}_R(X)$, for all complexes $X$ imply $R$ has finite global dimension?
This question was settled for Noetherian rings by Iacob and Iyengar  in~\cite{Iacob-Iyengar-regular}. %(using a revised language from~\cite{avramov-foxby-halperin-preprint})
They showed that when $R$ is Notherian, the two dimensions coincide if and only if $R$ is a regular ring.  We show in the current paper that the same is true for the larger class of all coherent rings. Our definition of \emph{regular} is the standard one often applied to a general ring; see Definition~\ref{definition-regular-ring}. We note that any ring of finite global dimension is regular, but there are coherent regular rings of infinite global dimension. In fact, among the class of coherent rings, the same statements are true if we replace `global dimension' with \emph{weak global dimension} (i.e. \emph{Tor-dimension}); see Example~\ref{example-regular-coherent-rings}.

Much of the essentials for the work in this paper were already worked out in~\cite{Iacob-Iyengar-regular}, but  the statements we give here for coherent rings give a more complete picture and unify various results from~\cite{Iacob-Iyengar-regular}.
Ultimately, given a (left) coherent ring $R$, we prove a slew of statements that are equivalent to the statement that $R$ is (left) regular. See Theorem~\ref{prop-characterize-fp}, Proposition~\ref{prop-regular-coderived-char}, and Theorems~\ref{them-coherent-regular-characterizations}--\ref{them-coherent-regular-characterizations-flat}. 

The results of this paper rely on some facts proved by Stovicek in~\cite{stovicek-purity}. In particular we use that the homotopy category of injective $R$-modules is compactly generated whenever $R$ is a coherent ring. We reprove the needed results here since we are only interested in modules over a ring, and the arguments we have found for this special case are easier than the general ones in~\cite{stovicek-purity}. The role of FP-injective modules is crucial. An $R$-module $M$ is called \emph{FP-injective}, or \emph{absolutely pure}, if $\Ext^1_R(F,M)=0$ for all finitely presented $R$-modules $F$.

This paper began after we proved Theorem~\ref{prop-characterize-fp}. We subsequently learned that, according to Glaz's book~\cite[Theorem~6.2.1]{glaz}, the equivalence of conditions (1) and  (2) of the proposotion goes back to Quentel~\cite{quentel}.
We also note that the forthcoming book~\cite[\S20.2]{lars-henrik-book}  contains further information on commutative Noetherian regular rings, including different proofs of many results from~\cite{Iacob-Iyengar-regular}. 

\

\emph{Notation and Terminology}:
Throughout the paper, $R$ denotes a ring with identity. We will favor the left side. In particular, when it is not explicitly stated, the words `module' and `ideal'  should be interpreted to mean left $R$-module and left ideal.  The abelian category of all (left) $R$-modules is denoted by  $R\textnormal{-Mod}$, and $\ch$ denotes the associated category of unbounded chain complexes of $R$-modules. We use homological notation for our chain complexes, so the differentials, which we denote by  $d$, lower the degree. For a given $R$-module $M$, we let $S^n(M)$ denote the chain complex consisting of $M$ concentrated in degree $n$, and 0 elsewhere. We let $D^n(M)$ denote the complex consisting of $1_M \colon M \xrightarrow{} M$ concentrated in degrees $n$ and $n-1$, and 0 elsewhere.

\section{Coherent regular rings}

\begin{definition}[Bertin \cite{bertin-coherent-regular}]\label{definition-regular-ring}
A  ring $R$ is \textbf{left regular}  if every finitely generated (left) ideal $I\subseteq R$ has $\text{pd}_R(I) < \infty$, that is, finite projective dimension. A  \textbf{right regular} ring is defined similarly and we say that $R$ is \textbf{regular} if it is both left and right regular.
%Say that $R$ is \emph{left FP-regular} if every finitely generated left ideal has a finite resolution by finitely generated projective left $R$-modules. Define \emph{right FP-regular} rings analogously. We  say that $R$ is \emph{FP-regular} if it is both left and right FP-regular.
\end{definition}

Recall that a ring $R$ is said to be  \textbf{left coherent} (resp. \textbf{right coherent})  if every  finitely generated left (resp. right)  ideal is finitely presented. A ring is called \textbf{coherent} if it is both left and right coherent. Any Noetherian ring is coherent. Example~\ref{example-regular-coherent-rings} provides examples of (not necessarily Noetherian) coherent regular rings.

%\emph{We will favor the left side. In particular, when it is not explicitly stated, the words `module' and `ideal'  should be interpreted to mean left $R$-modules.}

It is well-known that a ring is left coherent if and only if the full subcategory of all finitely presented $R$-modules is abelian. This is equivalent to saying that the kernel of any epimorphism between finitely presented modules is always again finitely presented.
Thus a ring $R$ is left coherent and left regular if and only if each finitely generated ideal has a finite resolution by finitely generated projective modules.  Such a ring will be called a  \textbf{left coherent regular} ring and the following theorem provides further characterizations.

%Recall that a chain complex of modules is said to be \textbf{perfect} if it is quasi-isomorphic to a bounded chain complex of finitely generated projective modules.

\begin{theorem}[Characterizations of coherent regular rings]\label{prop-characterize-fp}
The following are equivalent for a  general ring $R$.
\begin{enumerate}
\item $R$ is left coherent regular.
\item  $R$ is left coherent and every finitely presented module $M$ has $\text{pd}_R(M) < \infty$.
\item Each finitely generated left ideal has a finite resolution by finitely generated projective modules. 
\item Each finitely presented left $R$-module has a finite resolution by finitely generated projective  modules.
\item Every bounded chain complex of finitely presented modules  is  perfect. Said another way, each finitely presented chain complex is quasi-isomorphic to a finitely generated complex of projectives.
\end{enumerate}
\end{theorem}

  Of course, \emph{right coherent regular} rings are defined similarly, and we  say $R$ is a \textbf{coherent regular} ring if it is both left and right coherent regular.

\begin{proof}
The implication (1)$\impliedby$(2) is obvious. To prove
(1)$\implies$(2), % assume $R$ is (left) FP-regular. Then clealy any finitely generated  ideal is finitely presented. So $R$ is (left) coherent. Now
 let $M$ be any finitely presented $R$-module. Then there is a surjective homomorphism $R^n \xrightarrow{} M$ where $R^n$ is a finitely generated free module and such that its kernel is finitely generated (in fact, it's even finitely presented). So to show $\text{pd}_R(M) < \infty$, it is enough to prove: \emph{Any finitely generated submodule $S \subseteq  R^n$ has $\text{pd}_R(M) < \infty$.} To prove this, let $\pi \colon R^n \xrightarrow{} R$ denote the projection onto the $n^{\textnormal{th}}$ coordinate. Then given a finitely generated $S\subseteq R^n$,  there is a commutative diagram as below, with exact rows and such that each vertical arrow is just an inclusion. (Here we are identifying $R^{n-1}$ with the submodule of $R^n$ consisting of the first $n-1$ coordinates.)
 %We prove this assertion by induction on $n$; the case $n=1$ is true by the definition of coherent and regular. Now assume that any finitely generated submodule of $R^{n-1}$ has finite projective dimension.
 $$\begin{CD}
     0   @>>> R^{n-1}\cap S    @>>>   S    @>>>   \pi(S)   @>>>    0 \\
    @.        @VVV    @VVV      @VVV   @.\\
     0   @>>>  R^{n-1} @>>>  R^n    @>\pi>>   R    @>>>    0 \\
\end{CD}$$  Since $R$ is regular, the finitely generated ideal $\pi(S)\subseteq R$  must  have finite projective dimension. Since $R$ is coherent, $\pi(S)$ must  be finitely presented and hence $R^{n-1}\cap S$ is finitely generated as well.  So by induction on $n$ we may assume that $R^{n-1}\cap S$ is also of finite projective dimension. Then $S$ is an extension of two modules of finite projective dimension and is itself of finite projective dimension.

Now we prove (2)$\implies$(5). If (2) is true, then any complex of the form $S^n(F)$, where $F$ is finitely presented, is a perfect complex. Indeed the coherence hypothesis allows us to take a finite resolution $P_* \xrightarrow{\epsilon} F \xrightarrow{} 0$ with each $P_i$ a finitely generated projective. This determines a surjective quasi-isomorphism $P_* \xrightarrow{\epsilon} S^n(F)$, proving $S^n(F)$ is a perfect complex. Now we note that each bounded chain complex of finitely presented modules is inductively built up from a finite number of extensions of such complexes of the form $S^n(F)$. At each step, the extension is a short exact sequence, and thus induces a triangle in the derived category, $D(R)$.  By induction,  2 out of 3 complexes in the triangle are perfect, so the third one must also be perfect since the perfect complexes form a triangulated subcategory of $D(R)$.

We now show (5)$\implies$(4).
Let $F$ be a finitely presented module. Then  $S^0(F)$  is quasi-isomorphic to a bounded complex $P$ of finitely generated projective modules. %\\ $P = 0 \rightarrow P_n \rightarrow \cdots \rightarrow P_1 \xrightarrow{f_1} P_0 \rightarrow 0$ to the complex $0 \rightarrow F \rightarrow 0$ (with $F$ in the zeroth place).
 So $H_i(P) = 0$ for all $i \neq 0$, and $F \cong H_0(P) = Z_0P/B_0P$. Moreover, since $P$ is bounded below, $Z_iP$ is a direct summand of $P_i$, for all $i \leq 0$. In particular, $Z_0P$ is finitely generated projective. Hence we have an exact complex$$0 \rightarrow P_n \xrightarrow{d_n} \cdots \rightarrow P_2 \xrightarrow{d_2} P_1 \xrightarrow{d'_1} Z_0P \xrightarrow{} H_0(P) \rightarrow 0.$$
 This is a finite resolution of $F \cong H_0(P)$ by finitely generated projective modules.

Statements (1) and (3) are clearly equivalent as commented previously, before the statement of the theorem. Similarly,  (2)$\implies$(4). 

It remains to show (4)$\implies$(2).  So assume that each finitely presented $R$-module has a finite resolution by finitely generated projective  modules. %Since finitely generated projective modules are of type $FP_{\infty}$, this implies that each finitely presented module is of type $FP_{\infty}$. It follows that $R$ must be left coherent. See~\cite[Props. 2.1(2) and 2.3]{bravo-gillespie-hovey} for details. 
We claim that $R$ must be left coherent. So let $I$ be a finitely generated left ideal and the goal is to show that $I$ is in fact finitely presented. Certainly $R/I$ is finitely presented and so it has a finite resolution by finitely generated projective  modules. Note that each cycle module in any such resolution is necessarily finitely presented too. In particular there must be a short exact sequence $0 \xrightarrow{} F \xrightarrow{} P \xrightarrow{} R/I \xrightarrow{}  0$ where $P$ is finitely generated projective and $F$ is finitely presented. By Schanuel's Lemma we have that $I \bigoplus P \cong F \bigoplus R$. Since $F$ and $R$ are finitely presented it follows that  $I \bigoplus P$, and hence $I$, is finitely presented. 
\end{proof}

Sarah Glaz has studied commutative coherent rings, including those with the regularity condition.  Most of the  examples below can be found in her work. 

\begin{example}\label{example-regular-coherent-rings}
The following are examples of (left) regular coherent  rings.
\begin{enumerate}
\item Denote the weak global dimension (i.e. Tor-dimension) of a ring $R$ by $\text{w.dim}(R)$. If $R$ is left coherent, then by~\cite[Theorem~3.3]{stenstrom-fp} we have
$$\text{w.dim}(R) = \text{sup}\{ \text{pd}_R(F) \,|\, F \text{ finitely presented left } R\text{-module} \}.$$
It follows that any left coherent ring with $\text{w.dim}(R) < \infty$ is  left  regular.

\item  It is sometimes asserted that `von Neumann regular rings are unrelated to the usual notion of regular rings'. However, von Neumann regular rings are characterized  as  those rings $R$ such that $\text{w.dim}(R) =0$.  Any such ring is necessarily a (two-sided) coherent regular ring.

 %We point out that any von Neumann regular ring must be (left) semihereditary~\cite[Lemma~4.2.8]{weibel}.

A ring $R$ is called \emph{(left) semihereditary} if all finitely generated (left) ideals of $R$ are projective, and this happens if and only if all finitely generated submodules of projectives are again projective. Such a ring is automatically left coherent and has $\text{w.dim}(R)\leq 1$ \cite[Theorem~4.2.19]{glaz}.

\item There exist coherent regular rings of infinite weak dimension. For instance, let $k$ be a field. Then the polynomial ring $k[x_1, x_2, \dots]$ is not Notherian. However, it is a (commutative)  coherent regular ring of infinite weak dimension; see~\cite[page 202]{glaz}.  Similarly, the power series ring $k[[x_1, x_2, \dots]]$ in infinitely many variables is a (local) coherent regular ring of infinite weak dimension; \cite[\S5 example]{glaz-historical-persp}.

\item Let $R$ be a commutative coherent regular ring. Then the localization $R_{\mathfrak{p}}$ at any prime ideal $\mathfrak{p}$ is also a coherent regular ring and even a domain; see~\cite[Them.~6.2.3 and Cor.~6.2.4]{glaz}.
\end{enumerate}
\end{example}

\section{Complexes of FP-injective and FP-projective modules}

Throughout this section, $R$ will be a left coherent ring. The class $\class{S}$ of all finitely presented modules is essentially small, and closed under kernels of epimorphisms. So  it cogenerates,  for example  by~\cite[Lemma~3.6]{gillespie-models-for-hocats-of-injectives}, a complete hereditary cotorsion pair
$$(\class{FP}, \class{FI}) := (\leftperp{(\rightperp{\class{S}})}, \rightperp{\class{S}}).$$
The modules in $\class{FI}$ are called \textbf{FP-injective} (or \textbf{absolutely pure}) and the modules in $\class{FP}$ are called \textbf{FP-projective}.\footnote{In~\cite[Theorem~4.1.6(b)]{trlifaj-book},  the class $\class{FP}$ is denoted by $\class{FS}$,   for a module is FP-projective if and only if it is a direct \emph{summand} of a transfinite extension of \emph{finitely} presented modules.}

By way of a sequence of easy lemmas, we will now prove a special case of a result of Stovicek; see~\cite[Props.~6.11 and~B.7]{stovicek-purity}. The proof we give here is similar, but we are able to simplify some points since we are only dealing with $R$-modules.

\begin{lemma}\label{lemma-restriction-cot-pairs}
Consider the fully exact subcategory $\class{FI}$, of all FP-injective modules along with its inherited Quillen exact structure.  Then $\class{FP} \cap \class{FI}$ is the class of projective objects of $\class{FI}$, and the injective objects of $\class{FI}$ coincide with the usual injective modules. Moreover:
\begin{enumerate}
\item The complete cotorsion pair $(\class{FP}, \class{FI})$ in $R$-Mod restricts to a complete cotorsion pair $(\class{FP}\cap\class{FI}, \class{FI})$ in the exact category $\class{FI}$; it is the canonical projective cotorsion pair in $\class{FI}$.
\item The canonical injective cotorsion pair $(All, \class{I}nj)$ in $R$-Mod restricts to a complete cotorsion pair $(\class{FI},  \class{I}nj)$ in the exact category $\class{FI}$; it is the canonical injective cotorsion pair in $\class{FI}$.
\end{enumerate}
\end{lemma}

\begin{proof}
The easy verification is left to the reader. Note that the hereditary property of $(\class{FP}, \class{FI})$  implies that the cokernel of any monomorphism $A \xrightarrow{} I$ with $A \in \class{FI}$ and $I$ injective has cokernel $I/A \in \class{FI}$.
\end{proof}

\begin{definition}\label{definition-admissible-generators}
Let $\cat{A}$ be an abelian category and $\cat{B}\subseteq \cat{A}$ a full subcategory closed under extensions and direct summands. Considering $\cat{B}$  along with its naturally inherited Quillen exact structure, by a set  % = \{U_i\}_{i\in I}$
 of \textbf{admissible generators} for $\cat{B}$ we mean a set of objects $\class{U}\subseteq \class{B}$ satisfying the following condition: Any $\cat{B}$-morphism $f \colon A \xrightarrow{} B$ is necessarily  a $\cat{B}$-admissible epic (i.e.  it is onto  and has $\Ker{f}\in\cat{B}$) whenever  $f_* \colon  \Hom_{\cat{A}}(U,A) \xrightarrow{} \Hom_{\cat{A}}(U,B)$ is an epimorphism of abelian groups for all $U \in \class{U}$.
\end{definition}

Assuming $\cat{A}$ is cocomplete and $\cat{B}$ is closed under coproducts,  $\class{U}$  is a set of admissible generators for $\cat{B}$  if and only if each object $B\in\cat{B}$ is the target of a $\cat{B}$-admissible epic whose domain is  a coproduct of elements of $\class{U}$; see~\cite[Prop.~9.27]{gillespie-book}.

\begin{lemma}\label{lemma-FP-injective-generators}
Let $\kappa \geq |R|$ be an infinite cardinal.
Let $\class{U}$ denote a set of isomorphism representatives for the class of all FP-injective modules $A$ having cardinality $|A|\leq \kappa$. Then $\class{U}$ is a set of admissible generators for $\class{FI}$, the fully exact subcategory of all FP-injective modules.
\end{lemma}

\begin{proof}
Let $f \colon A \xrightarrow{} B$ be a morphism of FP-injective modules with the property that $f_* \colon  \Hom_R(U,A) \xrightarrow{} \Hom_R(U,B)$ is an epimorphism for all $U \in \class{U}$. We will show that $f$ must be a pure epimorphism, whence $K := \Ker{f}$ is also FP-injective, proving $f$ is an admissible epic in $\class{FI}$.
Now to show $f$ is a pure epimorphism means to show that $f_* \colon  \Hom_R(F,A) \xrightarrow{} \Hom_R(F,B)$ is an epimorphism for all finitely presented modules $F$. We note that $|F| \leq \kappa$ for any such $F$. Then using a well-known but nontrivial set theoretic fact about purity, see~\cite[Lemma~5.3.12]{enochs-jenda-book},  we have the following property: Given any homomorphism $\theta \colon F \xrightarrow{} B$, there exists a pure $P\subseteq B$ with $|P|\leq\kappa$ and $\im{\theta} \subseteq P$. Since $P \cong U \in \class{U}$, it follows from the hypothesis that the inclusion $P \subseteq B$ lifts over $f$. Hence there is a homomorphism $t \colon P \xrightarrow{} A$ such that $ft = \theta$.
\end{proof}

\begin{lemma}\label{lemma-FP-injective-cogenerators}
For each ideal $I\subseteq R$, choose a short exact sequence in $R$-Mod
$$0 \xrightarrow{}  R/I \xrightarrow{}  A_I \xrightarrow{}  C_I \xrightarrow{}  0$$
where $A_I \in \class{FI}$ and $C_I\in\class{FP}$.
Then the canonical injective cotorsion pair $(\class{FI},  \class{I}nj)$ in the exact category $\class{FI}$ (see Lemma~\ref{lemma-restriction-cot-pairs}) is cogenerated by the set $\{A_I\}_{I\subseteq R}$.
Therefore, taking $\class{U}$ as in Lemma~\ref{lemma-FP-injective-generators}, the set $\class{S} := \class{U}\cup \{A_I\}_{I\subseteq R}$  cogenerates $(\class{FI},  \class{I}nj)$ and also contains  a set of admissible  generators for the exact category $\class{FI}$.
\end{lemma}

\begin{proof}
For any $A\in\class{FI}$,  by applying  $\Ext^i_R(-,A)$, we obtain a short exact sequence
$$0 = \Ext^1_R(C_I,A) \xrightarrow{}  \Ext^1_R(A_I,A) \xrightarrow{}  \Ext^1_R(R/I,A) \xrightarrow{}  \Ext^2_R(C_I,A) = 0$$
and hence an  isomorphism $\Ext^1_R(R/I,A) \cong \Ext^1_R(A_I,A)$ for each ideal $I\subseteq R$.  Note too that since $A_I, A  \in \class{FI}$, the Ext group  $\Ext^1_R(A_I,A)$ coincides with  $\Ext^1_{\class{FI}}(A_I,A)$, the Yoneda Ext group relative to the exact structure inherited by $\class{FI}$. Thus $A\in\class{FI}$ is in ${\{A_I\}}^{\perp_{\class{FI}}}$ if and only if it is in ${\{R/I\}}^{\perp}$.
By Baer's criterion, this is the case if and only if $A$ is injective. This proves that $\{A_I\}_{I\subseteq R}$ is a cogenerating set for the canonical injective cotorsion pair $(\class{FI},  \class{I}nj)$ of Lemma~\ref{lemma-restriction-cot-pairs}(2).
\end{proof}

We let $\dwclass{I}$ denote the class of all complexes of injective modules (degree-wise injective complexes). We let $\class{W}_{\textnormal{co}}$ denote the class of all complexes $W$ such that any chain map $W \xrightarrow{} I$ is null homotopic whenever $I \in \dwclass{I}$. Following Positselski, we call such complexes   \textbf{coacyclic} (in the sense of Becker~\cite{becker}). All contractible complexes are coacyclic, and all coacyclic complexes are exact.

\begin{proposition}\label{prop-coacyclic-fp-injectives}
Let $R$ be a left coherent ring and $\class{S} := \class{U}\cup \{A_I\}_{I\subseteq R}$ be the cogenerating set for $(\class{FI},  \class{I}nj)$, from Lemma~\ref{lemma-FP-injective-cogenerators}. Then the set
$$D^n(\class{S}) :=\{\, D^n(S) \,|\,  S \in \class{S}, n\in\Z   \,\}$$
is a set of admissible generators for $\cha{FI}$, the exact category of all chain complexes of FP-injective modules along with its naturally inherited exact structure. Letting $\class{I} := \class{I}nj$ denote the class of injective modules, $D^n(\class{S})$ cogenerates a complete cotorsion pair $(\leftperp{\dwclass{I}}, \dwclass{I})$, in $\cha{FI}$. Moreover, we have
$$(\leftperp{\dwclass{I}}, \dwclass{I}) = (\class{W}_{\textnormal{co}} \cap \dwclass{FI}, \dwclass{I}).$$
That is, $(\leftperp{\dwclass{I}}, \dwclass{I})$ coincides with the restriction of the cotorsion pair $(\class{W}_{\textnormal{co}}, \dwclass{I})$ in $\ch$, to the fully exact subcategory $\cha{FI}$.
It follows that each complex of $\class{W}_{\textnormal{co}} \cap \dwclass{FI}$ is a direct summand of a transfinite extension of complexes in $D^n(\class{S})$.
\end{proposition}

\begin{proof}
The set $D^n(\class{U}) :=\{\, D^n(U) \,|\,  U \in \class{U}, n\in\Z   \,\}$ is a set of admissible generators for the exact category $\cha{FI}$.  Indeed this follows from Definition~\ref{definition-admissible-generators} and Lemma~\ref{lemma-FP-injective-generators} along with the natural isomorphism $$\Hom_{\ch}(D^n(U),X)\cong \Hom_R(U,X_n).$$

Therefore, the set $D^n(\class{S})$ cogenerates a (functorially) complete cotorsion pair, $(\leftperp{\dwclass{I}}, \dwclass{I})$, relative to the exact category $\cha{FI}$. Moreover,  $\leftperp{\dwclass{I}}$ equals the class of all direct summands of transfinite extensions of complexes in $D^n(\class{S})$; here we are using~\cite[Corollary~2.15(2)]{saorin-stovicek}, or see~\cite[Theorem~9.34(2)]{gillespie-book}.
 
We are claiming that $\leftperp{\dwclass{I}}$, which again denotes the left Ext orthogonal in $\cha{FI}$, coincides with $\class{W}_{\textnormal{co}} \cap\, \dwclass{FI}$, the class of coacyclic complexes of FP-injective modules. To see this, one easily checks directly that the (complete) cotorsion pair $(\class{W}_{\textnormal{co}}, \dwclass{I})$ in $\ch$, restricts to a (complete)  cotorsion pair $(\class{W}_{\textnormal{co}} \cap \dwclass{FI}, \dwclass{I})$ in $\cha{FI}$.  So it must be the case that  $\leftperp{\dwclass{I}}=\class{W}_{\textnormal{co}} \cap \dwclass{FI}$.
%So it must be the case that each complex of $\class{W}_{\textnormal{co}} \cap \dwclass{FI}$ is a direct summand of a transfinite extension of complexes in $D^n(\class{S})$; again see~\cite[Corollary~2.15(2)]{saorin-stovicek}, or~\cite[Theorem~9.34(2)]{gillespie-book}.
\end{proof}

The following is the special case we need of Stovicek's result~\cite[Prop.~6.11]{stovicek-purity}.

\begin{corollary}\label{cor-coacyclic-fp-injectives}
Let $R$ be a left coherent ring.
Let $\dwclass{FI}$ be the class of all degreewise FP-injective complexes, and let $\tilclass{FI}$ be the class of all exact complexes whose cycle modules are each FP-injective.  Then $\tilclass{FI} =  \class{W}_{\textnormal{co}}\cap \dwclass{FI}$ where $\class{W}_{\textnormal{co}}$ is the class of all coacylic complexes.
\end{corollary}

\begin{proof}
We have $(\subseteq)$ because each complex of $\tilclass{FI}$ is pure exact and $\class{W}_{\textnormal{co}}$ contains all pure exact complexes.
The reverse containment  $(\supseteq)$ follows from the last statement of Proposition~\ref{prop-coacyclic-fp-injectives} and the fact that the class of FP-injective modules is closed under extensions, direct limits and pure submodules, by~\cite{stenstrom-fp}. Indeed each complex of $\class{W}_{\textnormal{co}} \cap \dwclass{FI}$ is a direct summand of a transfinite extension of complexes in $D^n(\class{S})$. But any such  transfinite extension is necessarily a degreewise pure extension of pure exact complexes with FP-injective components, and so is itself a pure exact complex with FP-injective components, hence it is in $\tilclass{FI}$. Moreover $\tilclass{FI}$ is closed under direct summands. Therefore, $\class{W}_{\textnormal{co}} \cap \dwclass{FI} \subseteq \tilclass{FI}$.
\end{proof}

\subsection{Stovicek's FP-injective model structure for the coderived category}

By~\cite{gillespie}, the cotorsion pair $(\class{FP}, \class{FI})$ lifts to two hereditary cotorsion pairs on $\ch$, one of which gets denoted $(\dgclass{FP}, \tilclass{FI})$, where $\tilclass{FI}$ is the class of all exact complexes $Y$  with cycles $Z_nY\in\class{FI}$. However, it was shown in~\cite[Corollary~4.8]{bazzoni-hrbek-positselski-fp-projective-periodicity} that every chain map $f \colon X \xrightarrow{} Y$ is null homotopic whenever $Y\in\tilclass{FI}$ and each $X_n\in\class{FP}$. In other words, $\dgclass{FP}=\dwclass{FP}$, where $\dwclass{FP}$ denotes the class of all complexes of FP-projective modules.
This is not just true for modules over a left coherent ring, but over any ring $R$. For the reader's convenience we include a proof of this result.

\begin{proposition}\cite[Corollary~4.8]{bazzoni-hrbek-positselski-fp-projective-periodicity}\label{prop-complexes-fp-projectives}
Let $R$ be any ring. Then any chain map $f \colon X \xrightarrow{} Y$ is null homotopic whenever $Y\in\tilclass{FI}$ and each $X_n\in\class{FP}$. In other words, $\dgclass{FP}=\dwclass{FP}$.% Intersecting with the class of exact complexes we have $\tilclass{FP}=\exclass{FP}$.
\end{proposition}

\begin{proof}
By a variation of Eilenberg's trick, we may assume from the start that each $X_n$ is a transfinite extension of finitely presented modules; see~\cite[Lemma~3.2]{gillespie-ac-pure-proj}. Then by~\cite[Prop.~4.3]{stovicek-deconstructible}, with $\kappa=\aleph_0$, the complex $X$ is a transfinite extension of bounded above complexes of finitely presented modules. Hence by Eklof's lemma we may even assume that each $X$ is a bounded above complex of finitely presented modules. But in this case the statement of the proposition is immediate from~\cite[Lemma~3.2]{gillespie-ac-pure-proj} because any complex $Y\in\tilclass{FI}$ is easily seen to be pure exact. %But statement of the proposition was shown for bounded above complexes of finitely presented modules in~\cite[Lemma~3.2]{gillespie-ac-pure-proj}.
\end{proof}

\begin{proposition}\cite{stovicek-purity}\label{prop-coderived-fp-injectives}
Let $R$ be a left coherent ring. Then the triple of classes
$$\mathfrak{M}^{FPinj}_{\textnormal{co}} = (\dwclass{FP}, \class{W}_{\textnormal{co}}, \dwclass{FI})$$
represents an hereditary abelian model structure on $\ch$ and satisfies:
\begin{enumerate}
\item Its homotopy category identifies with the coderived category of $R$, i.e.,  $$\textnormal{Ho}(\mathfrak{M}^{FPinj}_{\textnormal{co}})=\class{D}_{\textnormal{co}}(R).$$
\item The canonical functor $\gamma_{\textnormal{co}} \colon \ch \xrightarrow{} \class{D}_{\textnormal{co}}(R)$ preserves coproducts.
\item $\mathfrak{M}^{FPinj}_{\textnormal{co}}$ is  finitely generated, so $\class{D}_{\textnormal{co}}(R)$ is compactly generated.
\item The compact objects of $\class{D}_{\textnormal{co}}(R)$ are precisely the complexes that are isomorphic to a finitely presented object of $\ch$. I.e., isomorphic in the coderived category to a bounded complex of finitely presented modules.
\end{enumerate}
\end{proposition}

\begin{proof}
The idea of the proof is to show that each of the two cotorsion pairs making up the Hovey triple are cogenerated by sets of finitely presented complexes.
First, consider the set $\class{S} = \{S^n(F) \,|\, n\in\Z\}$, where $F$ ranges through a set of isomorphism representatives for all the finitely presented modules. It cogenerates a complete cotorsion pair, and using Proposition~\ref{prop-complexes-fp-projectives} we see that it is exactly $(\dwclass{FP},\tilclass{FI})$.   In light of Corollary~\ref{cor-coacyclic-fp-injectives}, we conclude that $(\dwclass{FP}, \tilclass{FI})=(\dwclass{FP}, \class{W}_{\textnormal{co}}\cap \dwclass{FI})$ is cogenerated by $\class{S}$, which we note is a set of finitely presented complexes.

On the other hand, consider the  set $\class{T} = \{D^n(F) \,|\, n\in\Z\}$, where again $F$ ranges through a set of isomorphism representatives for all the finitely presented modules. It too cogenerates a complete cotorsion pair, and by~\cite[Prop.~4.4]{gillespie-degreewise-model-strucs} it is exactly $(\leftperp{\dwclass{FI}},\dwclass{FI})$, where $\dwclass{FI}$ is the class of all complexes that are degreewise FP-injective. We claim that $\leftperp{\dwclass{FI}} = \dwclass{FP}\cap\class{W}_{\textnormal{co}}$. Indeed the containment $\leftperp{\dwclass{FI}} \subseteq \dwclass{FP}\cap\class{W}_{\textnormal{co}}$ is rather immediate. On the other hand, let $X\in\dwclass{FP}\cap\class{W}_{\textnormal{co}}$. We may write a short exact sequence $0 \xrightarrow{}  Y \xrightarrow{}  Z \xrightarrow{}  X \xrightarrow{}  0$ with $Z \in\leftperp{\dwclass{FI}}$ and $Y\in \dwclass{FI}$. Since $\leftperp{\dwclass{FI}} \subseteq \class{W}_{\textnormal{co}}$, and  $\class{W}_{\textnormal{co}}$ is thick, we conclude $Y\in\class{W}_{\textnormal{co}}\cap\dwclass{FI}=\tilclass{FI}$. Now since $(\dwclass{FP}, \tilclass{FI})$ is a cotorsion pair the short exact sequence must split, and so $X$ is a direct summand of $Z$, proving $X \in\leftperp{\dwclass{FI}}$.
We conclude that $(\dwclass{FP}\cap \class{W}_{\textnormal{co}},  \dwclass{FI})$ is cogenerated by $\class{T}$ which again is a set of finitely presented complexes.
This completes the proof that $\mathfrak{M}^{FPinj}_{\textnormal{co}} = (\dwclass{FP}, \class{W}_{\textnormal{co}}, \dwclass{FI})$ is an  hereditary  abelian model structure. Because the cogenerating sets consist of finitely presented complexes, it follows from~\cite[Corollary~8.45]{gillespie-book} that  $\textnormal{Ho}(\mathfrak{M}^{FPinj}_{\textnormal{co}})$ is compactly generated and that $\class{S}$ is a set of compact weak generators. Since the trivial objects are the coacyclic complexes, this is a model for the coderived category, $\class{D}_{\textnormal{co}}(R)$. The model structure is finitely generated in the sense of~\cite[Ch.~7]{hovey-model-categories}.

The precise statement about preservation of coproducts is this: Given a collection of complexes $\{X_i\}_{i\in I}$, let $X_i \xrightarrow{\eta_i} \bigoplus_{i\in I} X_i$ denote the canonical injections into a coproduct of $\{X_i\}_{i\in I}$ in $\ch$. Then $$(\bigoplus_{i\in I}X_i, \{\gamma_{\textnormal{co}}(\eta_i)\}_{i\in I})$$ is a coproduct of  $\{X_i\}_{i\in I}$ in  $\textnormal{Ho}(\mathfrak{M}^{FPinj}_{\textnormal{co}})=\class{D}_{\textnormal{co}}(R)$.
In the present case, this statement is an instance of~\cite[Prop.~5.30 or Lemma~8.39]{gillespie-book}.

Since $\class{S}$ consists of compact objects, and since each finitely presented complex (i.e. bounded complex of finitely presented modules) is inductively built up by finitely many extensions of complexes in $\class{S}$, it follows that each finitely presented complex is  compact as an object of $\textnormal{Ho}(\mathfrak{M}^{FPinj}_{\textnormal{co}})=\class{D}_{\textnormal{co}}(R)$. Moreover, any  bounded complex, or even bounded above complex,  is coacyclic  if and only if it is exact. This is because any bounded above exact complex is a transfinite extension of  complexes of the form $D^n(M)\in\class{W}_{\textnormal{co}}$. Thus one may check that the inclusion  $\textnormal{Ch}^b(R\textnormal{-mod})\hookrightarrow \ch$ of the full subcategory of finitely presented complexes, induces a fully faithful inclusion $\class{D}^b(R\textnormal{-mod}) \hookrightarrow  \class{D}_{\textnormal{co}}(R)$ of the bounded derived category of  $R\textnormal{-mod}$,  the abelian category of finitely presented modules. Therefore its essential image, i.e., the complexes in $\class{D}_{\textnormal{co}}(R)$ that are isomorphic to a finitely presented complex, is precisely the subcategory of compact objects of $\class{D}_{\textnormal{co}}(R)$. This follows from the well-known fact that in any compactly generated category, the full subcategory of compact objects equals the smallest thick subcategory containing any set of compact weak generators.
(Note that the argument here is in analogy to the one in~\cite[Cor.~6.13]{stovicek-purity} which shows  that any choice of a fibrant replacement equivalence functor $R \colon \textnormal{Ho}(\mathfrak{M}^{inj}_{\textnormal{co}}) \xrightarrow{} K(Inj)$ restricts to an equivalence between $\class{D}^b(R\textnormal{-mod})$ and the compact objects of $K(Inj)$.)
\end{proof}

\section{Homological dimensions of complexes over coherent regular rings}

Again we let $R$ be a left coherent ring and let
$\mathfrak{M}^{FPinj}_{\textnormal{co}} = (\dwclass{FP}, \class{W}_{\textnormal{co}}, \dwclass{FI})$ denote the FP-injective model structure for the coderived category.

\subsection{Injective and FP-injective characterizations of coherent regular rings}\label{subsec-FP-inj-coherent-rings}

Referring again to~\cite{gillespie}, the cotorsion pair $(\class{FP}, \class{FI})$ lifts to another hereditary cotorsion pair on $\ch$, this one denoted $(\tilclass{FP}, \dgclass{FI})$.  This cotorsion pair is also complete; for example, see~\cite[Lemma~10.47]{gillespie-book}.
A complex in the class $\dgclass{FI}$ will be called  \textbf{DG-FP-injective}\footnote{It is shown in~\cite{emmanouil-kaperonis-K-abs-pure} that these are precisely the K-absolutely pure complexes with FP-injective components.}. We also let $\dwclass{FI}$ (resp. $\exclass{FI}$) denote the class of all complexes (resp. all exact complexes) that are degreewise FP-injective.

\begin{lemma}\label{lemma-derived-fp-injectives-quotient-functor}
Let $R$ be a left coherent ring and let $\class{E}$ denote the class of all exact (i.e. acyclic) complexes of $R$-modules.
Then the triple of classes
$$\mathfrak{M}^{FPinj}_{\textnormal{der}} = (\dwclass{FP}, \class{E}, \dgclass{FI})$$
represents an hereditary abelian model structure on $\ch$ and satisfies:
\begin{enumerate}
\item Its homotopy category identifies with the usual derived category of $R$, i.e.,  $$\textnormal{Ho}(\mathfrak{M}^{FPinj}_{\textnormal{der}})=\class{D}(R),$$
 and the canonical functor $\gamma \colon \ch \xrightarrow{} \class{D}(R)$ preserves coproducts.
\item There is a unique functor $\bar{\gamma} \colon  \class{D}_{\textnormal{co}}(R) \xrightarrow{} \class{D}(R)$, %\textnormal{Ho}(\mathfrak{M}^{FPinj}_{\textnormal{co}})  \xrightarrow{} \textnormal{Ho}(\mathfrak{M}^{FPinj}_{\textnormal{der}})$,
called the \textbf{quotient functor}, satisfying $\gamma = \bar{\gamma} \circ \gamma_{\textnormal{co}}$.  Moreover, $\bar{\gamma}$ is  a triangulated functor and it preserves coproducts.
\end{enumerate}
\end{lemma}

\begin{proof}
We have $\tilclass{FP} = \dwclass{FP}\cap\class{E}$ and $\class{E}\cap\dgclass{FI} = \tilclass{FI}$,  by Proposition~\ref{prop-complexes-fp-projectives} along with~\cite[Corollary~3.13]{gillespie}. So the complete hereditary cotorsion pairs $(\tilclass{FP}, \dgclass{FI})$, and  $(\dwclass{FP}, \tilclass{FI})$ %of Proposition~\ref{prop-coderived-fp-injectives}, we conclude 
 imply that $\mathfrak{M}^{FPinj}_{\textnormal{der}}$ is an hereditary abelian model structure. 
Since the trivial objects  are precisely the exact complexes, it is a model for the usual derived category, $\class{D}(R)$.
As was the case for $\gamma_{\textnormal{co}}$ in Proposition~\ref{prop-coderived-fp-injectives}, it follows from~\cite[Corollary~8.45]{gillespie-book} that the canonical functor $\gamma \colon \ch \xrightarrow{} \textnormal{Ho}(\mathfrak{M}^{FPinj}_{\textnormal{der}})$ is compatible with coproducts.

The existence and uniqueness of the quotient functor $\bar{\gamma}$, and the fact that it is a triangulated functor, follows from the Triangulated Localization Theorem~\cite[Theorem~6.44]{gillespie-book}. The commutativity of the diagram $\gamma = \bar{\gamma} \circ \gamma_{\textnormal{co}}$ implies that the quotient functor $\bar{\gamma}$ preserves coproducts.
\end{proof}

Building on  the characterizations of left coherent regular rings given in Theorem~\ref{prop-characterize-fp}  we have the following proposition and theorems. The main idea  comes from~\cite[\S2]{Iacob-Iyengar-regular}, though we are substituting   $\textnormal{Ho}(\mathfrak{M}^{FPinj}_{\textnormal{co}}) = \class{D}_\textnormal{co}(R)$ for the (equivalent) homotopy category $K(Inj)$.

\begin{proposition}[Characterizations of coherent regular rings]\label{prop-regular-coderived-char}
Let $R$ be a left coherent ring. The following are equivalent:
\begin{enumerate}
\item $R$ is left regular.
 \setcounter{enumi}{5}
\item\label{quotient-func} The quotient functor $\bar{\gamma} \colon   \class{D}_{\textnormal{co}}(R)  \xrightarrow{} \class{D}(R)$ is a triangulated equivalence. In this case, $\class{D}_{\textnormal{co}}(R)  = \class{D}(R)$ and  $\bar{\gamma}$ is the identity functor.
\item\label{coacyc} The class $\class{W}_{\textnormal{co}}$ of all coacyclic complexes coincides with the class $\class{E}$ of all exact complexes.
\end{enumerate}
\end{proposition}

\begin{proof}
\ref{coacyc} $\Rightarrow$ \ref{quotient-func}. If $\class{W}_{\textnormal{co}}=\class{E}$, then clearly $\bar{\gamma}$ is the identity, so it's an equivalence. \\
\ref{quotient-func} $\Rightarrow$ \ref{coacyc}.
Evidently, $\Ker{\bar{\gamma}} =\class{E}$. But the kernel of any additive equivalence must equal the class of zero objects, and  $\class{W}_{\textnormal{co}}$ is the class of zero objects of $\class{D}_{\textnormal{co}}(R)$.

1 $\Leftrightarrow$ \ref{quotient-func}.  It is well-known that a coproduct preserving triangulated functor between compactly generated categories is an equivalence if and only if it restricts to an equivalence between their full subcategories of  compact objects; see~\cite[2.4]{Iacob-Iyengar-regular} and the references given there.
If $R$ is left coherent regular then, by Theorem~\ref{prop-characterize-fp},  each bounded complex of finitely presented modules  is  perfect. This implies that the restriction of $\bar{\gamma} \colon  \class{D}_{\textnormal{co}}(R) \xrightarrow{} \class{D}(R)$ to the (strictly) full subcategory $\class{D}^b(R\textnormal{-mod}) \hookrightarrow  \class{D}_{\textnormal{co}}(R)$, of bounded complexes of finitely presented modules, is an equivalence onto the (strictly) full subcategory $K^{b}(\textnormal{proj}) \hookrightarrow  \class{D}(R)$, of perfect complexes. It follows that $\bar{\gamma}$ is an equivalence since it is a coproduct preserving triangulated functor (Lemma~\ref{lemma-derived-fp-injectives-quotient-functor}) between compactly generated categories. Conversely, if $\bar{\gamma}$ is an equivalence then because we have $\gamma = \bar{\gamma} \circ \gamma_{\textnormal{co}}$, Proposition~\ref{prop-coderived-fp-injectives}(4) implies that each bounded complex of finitely presented modules  is  perfect. So by Theorem~\ref{prop-characterize-fp} $R$ must be left coherent regular.
\end{proof}

Now we can characterize coherent regular rings in terms of  complexes of FP-injective modules. First, recall that for a (left) coherent ring there is a well-behaved notion of the FP-injective dimension, which we denote by $\text{fpid}_R(M)$, of a module $M$. See~\cite[\S3]{stenstrom-fp} for details. Using the Avramov-Foxby approach from~\cite{avramov-foxby-hom-dims-complexes}, we may extend it to the \emph{FP-injective dimension} of a complex $X$:
$$\text{fpid}_R(X) := \text{min}\{\,n \in \Z \,|\, X \simeq I  \text{ for some } I \in \dgclass{FI} \text{ with } I_i = 0,  \ \forall \,i < -n\,\}.$$
Also, by replacing $\dgclass{FI}$ with  the class $\dwclass{FI}$ of all complexes of FP-injective modules, we  obtain the \emph{graded FP-injective dimension} of $X$, denoted $\text{gr-fpid}_R(X)$. Of course we have  $\text{gr-fpid}_R(X) \le \text{fpid}_R(X)$ since in general $\dgclass{FI} \subseteq \dwclass{FI}$.

\begin{theorem}[FP-injective characterizations of coherent regular rings]\label{them-coherent-regular-characterizations}
Let $R$ be a left coherent ring. The following are equivalent:
\begin{enumerate}
\item $R$ is left regular.
\setcounter{enumi}{7}
\item\label{item-fp-one}  Every chain complex of FP-injective modules is
  DG-FP-injective. That is, $\dwclass{FI} = \dgclass{FI}$.
\item\label{item-fp-two}  Every exact complex of FP-injective modules is categorically FP-injective, equivalently, pure exact. That is,  $\exclass{FI} = \tilclass{FI}$.
%\item Every exact complex of FP-injective modules is pure exact.
\item\label{item-fp-three} The class $\dgclass{FI}$ of DG-FP-injective complexes is closed under direct sums.
\item\label{item-fp-four} The class $\dgclass{FI}$ of DG-FP-injective complexes is closed  under direct limits.
\item\label{item-fp-five} $\text{gr-fpid}_R(X) = \text{fpid}_R(X)$ for all complexes of modules $X$.
\item\label{item-fp-six} $\text{gr-fpid}_R(M) = \text{fpid}_R(M)$ for all modules $M$.
\end{enumerate}
\end{theorem}

\begin{proof}
1 $\Leftrightarrow$ \ref{item-fp-one}. The model structures $(\dwclass{FP}, \class{E}, \dgclass{FI})$ and $(\dwclass{FP}, \class{W}_{\textnormal{co}}, \dwclass{FI})$, from Lemma~\ref{lemma-derived-fp-injectives-quotient-functor} and Proposition~\ref{prop-coderived-fp-injectives}, have the same cofibrant objects. It follows that  $\class{W}_{\textnormal{co}}=\class{E}$ if and only if $\dwclass{FI}=\dgclass{FI}$. (The ``if'' direction is an instance of~\cite[Prop.~3.2]{gillespie-recollement}.) So by Proposition~\ref{prop-regular-coderived-char},  $R$ is left  regular if and only if $\dwclass{FI} = \dgclass{FI}$.  \\
\ref{item-fp-one} $\Rightarrow$ \ref{item-fp-four}. This is immediate since the class of FP-Injective modules is closed  under direct limits when $R$ is left coherent~\cite[Them.~3.2(a)(d)]{stenstrom-fp}.\\
\ref{item-fp-four} $\Rightarrow$ \ref{item-fp-three}. It is clear that $\dgclass{FI}$ is closed under \emph{finite} direct sums. So this follows from the fact that any (infinite) direct sum may be expressed as the direct limit of its finite summands along with their natural inclusions.\\
\ref{item-fp-three} $\Rightarrow$ \ref{item-fp-one}. Let $I = \cdots \rightarrow I_1 \rightarrow I_0 \rightarrow I_{-1} \rightarrow \cdots$ be a complex of FP-injective modules. For each $n \in \Z$, set $I(n) = 0 \rightarrow I_n \rightarrow I_{n-1} \rightarrow I_{n-2} \rightarrow \cdots$. Being a bounded above complex of FP-injectives, each $I(n)$ is a DG-FP-injective complex. Clearly we have $I = \bigcup_{n \in \N} I(n)$, and this may also be expressed as $I = \varinjlim_{n\in\N} I(n)$.  So using the standard construction of a direct limit as the cokernel of a morphism between two coproducts, $I$ is the cokernel of a map $\theta \colon \bigoplus_{n\in\N}  I(n) \xrightarrow{}  \bigoplus_{n\in\N}  I(n)$, defined on each summand by the rule $e_n - e_{n+1}\circ i_n \colon I(n) \rightarrow\bigoplus_{n\in\N} I(n)$ where $e_n \colon I(n) \rightarrow\bigoplus_{n\in\N} I(n)$ denotes the canonical insertion into the coproduct and $i_n \colon I(n) \rightarrow I(n+1)$ is  the obvious inclusion map. Evidently, $\theta$ is a monomorphism, so there is a short exact sequence $$0 \rightarrow \bigoplus_{n\in\N}  I(n) \xrightarrow{\theta}  \bigoplus_{n\in\N}  I(n) \xrightarrow{} I \rightarrow 0.$$ 	%$(x_n) = (x_n - i_{n-1}(x_{n-1}))$, where $i_n\colon I(n) \rightarrow I(n+1)$ is the inclusion map.
Since the class of DG-FP-Injective complexes is closed under direct sums and cokernels of monomorphisms (\cite[Cor.~3.13]{gillespie}), it follows that $I$ is DG-FP-Injective.\\
 \ref{item-fp-one} $\Rightarrow$  \ref{item-fp-two}. First, we note that for the current case of a left coherent ring,  \cite[Prop.~2.6]{bravo-gillespie-abs-clean-level-complexes} shows  that $\tilclass{FI}$ is precisely the class of categorically FP-injective complexes. Second, we also note  that an exact complex of FP-injectives is pure exact if and only if it is in $\tilclass{FI}$. This follows easily from the fact that a submodule of an FP-injective is again FP-injective if and only if it is pure~\cite[Prop.~2.6(d)]{stenstrom-fp}.  Now let $F$ be an exact complex of FP-Injective modules. By hypothesis, $F$ is DG-FP-injective, so $F \in \mathcal{E} \cap \dgclass{FI}= \tilclass{FI}$, where $\class{E}$ is the class of exact complexes and the equality comes from~\cite[Cor.~3.13]{gillespie}. \\
 \ref{item-fp-two} $\Rightarrow$  \ref{item-fp-one}. Let $X$ be any complex of FP-Injective modules. Since $(\tilclass{FP}, \dgclass{FI})$ is a complete cotorsion pair, there is an exact sequence $0 \rightarrow C \rightarrow D \rightarrow X \rightarrow 0$ with $C \in \dgclass{FI}$ and with $D \in \tilclass{FP}$. Then $D$ is an exact complex of FP-injective modules. So by hypothesis, $D\in \tilclass{FI}=\mathcal{E} \cap \dgclass{FI}$. Since the cotorsion pair  $(\tilclass{FP}, \dgclass{FI})$ is hereditary (\cite[Cor.~3.13]{gillespie}) and since  $C, D\in\dgclass{FI}$,  it follows that $X$ is also DG-FP-injective. \\
%3 $\Rightarrow$ 4. By (3), all the cycles of any exact complex of FP-injective modules are also FP-injective modules. So, every exact complex of FP-injective modules is $Hom(C,-)$ exact, for any finitely presented module $C$. By [Emmanouil, On pure acyclic complexes], Proposition 2.2, it follows that every exact complex of FP-injectives is pure exact.\\
%4 $\Rightarrow$ 3. Let $F$ be an exact complex of FP-injective modules. Then $F^+$ is a contractible complex of flat modules ([Emmanouil, On pure acyclic complexes], Proposition 1.1), so $(Z_nF)^+$ is flat for all $n$. Since $R$ is coherent, this implies that $Z_n F$ is FP-injective, for all $n$. So $F \in \widetilde{FP-Inj}$.\\
%3 $\Rightarrow$ 5. Let $F$ be an exact complex of flat modules. Then $F^+$ is an exact complex of FP-injective modules. By (3), $Z_j(F)^+$ is FP-injective, and therefore $Z_j(F)^{++}$ is flat,  for all $j$. Since $Z_j(F)$ is a pure submodule of $Z_j(F)^{++}$, it is also flat.\\
%5 $\Rightarrow$ 3. If $C$ is an exact complex of FP-injectives then $C^+$ is an exact complex of flat modules, so $Z_j(C)^+$ is flat for all $j$. It follows that $Z_j(F)$ is FP-injective.\\
%5 $\Leftrightarrow$ 6 follows from Prop. 3.4 and Theorem 3.8 in the paper with Srikanth.
\ref{item-fp-one} $\Rightarrow$ \ref{item-fp-five} is clear. \\
\ref{item-fp-five} $\Rightarrow$ \ref{item-fp-six}. For a  module $M$, we set $\text{gr-fpid}_R(M) := \text{gr-fpid}_R(S^0(M))$ which by hypothesis equals $\text{fpid}_R(S^0(M))$. But it turns out that $\text{fpid}_R(S^0(M)) = \text{fpid}_R(M)$ (analogous to the corresponding fact for flat dimension~\cite[Remark 2.8.F]{avramov-foxby-hom-dims-complexes}). \\
\ref{item-fp-six} $\Rightarrow$ \ref{item-fp-two}.  Let $I = \cdots \rightarrow I_1 \xrightarrow{d_1} I_0 \xrightarrow{d_0} I_{-1} \rightarrow \cdots $ be an exact complex of FP-injective modules.
For each integer $k$, consider the truncated complex $${}_kI := \cdots \rightarrow I_{k+2} \rightarrow I_{k+1} \xrightarrow{d_{k+1}} I_k \rightarrow 0.$$ There is a quasi-isomorphism ${}_kI \rightarrow S^k(Z_{k-1}I)$  given by $d_k:I_k \rightarrow Z_{k-1}I$, and zeros everywhere else. This shows  $\text{gr-fpid}_R(Z_{k-1}I) =0$, and so by the hypothesis $\text{fpid}_R(Z_{k-1}I) =0$. Thus $Z_{k-1}I$ is FP-injective, for each integer $k$. 
\end{proof}

Next we give the analogous characterizations for complexes of injectives. The following result is the extension of Iacob-Iyengar's result~\cite[Theorem~1.1]{Iacob-Iyengar-regular} from Noetherian to coherent rings. Thus it extends to  coherent rings their answer to Foxby-Avramov's question from~\cite[Question~3.8]{avramov-foxby-hom-dims-complexes}.

\begin{theorem}[Injective characterizations of coherent regular rings]\label{them-coherent-regular-characterizations-injective}
Let $R$ be a left coherent ring. Then the  following are equivalent:
\begin{enumerate}
\item $R$ is left regular.
 \setcounter{enumi}{13}
\item\label{item-inj-one}  Every complex of injective modules is
  DG-injective. That is, $\dwclass{I} = \dgclass{I}$.
\item  Every exact complex of injective modules is  categorically injective, equivalently, contractible. That is,  $\exclass{I} = \tilclass{I}$.
%\item The class $\class{W}_{\textnormal{co}}$ of all coacyclic complexes coincides with the class $\class{E}$ of all exact complexes.
\item $\text{gr-id}_R(X) = \text{id}_R(X)$ for all complexes of modules $X$.
\item\label{item-inj-four}  $\text{gr-id}_R(M) = \text{id}_R(M)$ for all modules $M$.
\end{enumerate}
\end{theorem}

\begin{proof}
Conditions (\ref{item-inj-one})--(\ref{item-inj-four}) are equivalent by~\cite[Prop.~2.1]{Iacob-Iyengar-regular}. Using Theorem~\ref{them-coherent-regular-characterizations} it is enough to show $\exclass{FI} = \tilclass{FI}$ if and only if $\exclass{I} = \tilclass{I}$.
The ``only if'' part follows from Stovicek's result that a pure exact complex of injectives must be contractible. We prove the ``if'' part: Using that $(\class{W}_{\textnormal{co}}, \dwclass{I})$ is a complete cotorsion pair, we may embed any given exact complex  of FP-injective modules, $F$, into a complex $I$ of injective modules
$0 \xrightarrow{} F \xrightarrow{} I \xrightarrow{} W \xrightarrow{} 0$, with $W\in\class{W}_{\textnormal{co}}$. That is, $W$ is coacyclic. Then  $I$ must also be exact, so contractible by hypothesis, hence $I$ too is coacyclic. By the 2 out of 3 property, $F$ must also be in $\class{W}_{\textnormal{co}}$. But by Corollary~\ref{cor-coacyclic-fp-injectives}
we have $\class{W}_{\textnormal{co}}\cap \dwclass{FI}=\tilclass{FI}$, so  $F \in\tilclass{FI}$. This proves $F$ must be pure exact.
\end{proof}

\subsection{Projective and Flat characterizations of coherent regular rings}
Finally, we also have the projective and flat versions of the above results.
The following two theorems were proved by Iacob and Iyengar in~\cite[\S3]{Iacob-Iyengar-regular}. The only difference here is that we   explicitly tie all of their characterizations to the notion of a left coherent regular ring.

Up until now, all `modules' have been left $R$-modules, but note the side swapping here between \emph{left} coherent and \emph{right} $R$-modules.

\begin{theorem}[Projective characterizations of coherent regular rings--\S3 of~\cite{Iacob-Iyengar-regular}]\label{them-coherent-regular-characterizations-projective}
Let $R$ be a left coherent ring. Then the  following are equivalent:
\begin{enumerate}
\item $R$ is left regular.
\setcounter{enumi}{17}
\item\label{item-proj-one} Every chain complex of projective right $R$-modules is
  DG-projective. That is, $\dwclass{P} = \dgclass{P}$ for complexes of right $R$-modules.
\item  Every exact complex of projective right $R$-modules is  categorically projective, equivalently, contractible. That is,  $\exclass{P} = \tilclass{P}$  for right $R$-modules.
\item The class $\class{W}_{\textnormal{ctr}} := \rightperp{(\dwclass{P})}$ of all contraacyclic complexes  of right $R$-modules coincides with the class $\class{E}$ of all exact complexes.
\item $\text{gr-pd}_R(X) = \text{pd}_R(X)$ for all complexes of right $R$-modules $X$.
\item\label{item-proj-five}  $\text{gr-pd}_R(M) = \text{pd}_R(M)$ for all right $R$-modules $M$.
\end{enumerate}
\end{theorem}

\begin{theorem}[Flat characterizations of coherent regular rings--\S3 of~\cite{Iacob-Iyengar-regular}]\label{them-coherent-regular-characterizations-flat}
Let $R$ be a left coherent ring. Then the  following are equivalent:
\begin{enumerate}
\item $R$ is left regular.
\setcounter{enumi}{22}
\item\label{item-flat-one}  Every chain complex of flat right $R$-modules is
  DG-flat. That is, $\dwclass{F} = \dgclass{F}$.
\item  Every exact complex of flat right $R$-modules is  categorically flat, equivalently, pure exact. That is,  $\exclass{F} = \tilclass{F}$.
\item The class $\dgclass{F}$, of DG-flat complexes of right $R$-modules, is closed under products.
\item $\text{gr-fd}_R(X) = \text{fd}_R(X)$ for all complexes of right $R$-modules $X$.
\item\label{item-proj-five} $\text{gr-fd}_R(M) = \text{fd}_R(M)$ for all right $R$-modules $M$.
\end{enumerate}
\end{theorem}

\begin{proof}[Proof of Theorems~\ref{them-coherent-regular-characterizations-projective} and~\ref{them-coherent-regular-characterizations-flat}]
Combine the results of~\cite[\S3]{Iacob-Iyengar-regular} with the characterization of left coherent regular rings given in Theorem~\ref{prop-characterize-fp}.
\end{proof}

Again, there are versions of all of the above results for \emph{right} coherent regular rings. Combining, we obtain characterizations of two-sided coherent regular rings:

\begin{corollary}\label{cor-regular}
Let $R$ be a (two-sided) coherent ring. Then the following are equivalent: % where `$R$-module' may be taken to mean either `left $R$-module' or `right $R$-module'.
\begin{enumerate}
\item $R$ is (two-sided) regular, i.e, a coherent regular ring.
\item One of the conditions from Theorem~\ref{them-coherent-regular-characterizations} or Theorem~\ref{them-coherent-regular-characterizations-injective} holds for left  $R$-modules,  and  one of the conditions from Theorem~\ref{them-coherent-regular-characterizations-projective} or Theorem~\ref{them-coherent-regular-characterizations-flat} also holds for left $R$-modules.
\item One of the conditions from Theorem~\ref{them-coherent-regular-characterizations} or Theorem~\ref{them-coherent-regular-characterizations-injective} holds for right  $R$-modules, and  one of the conditions from Theorem~\ref{them-coherent-regular-characterizations-projective} or Theorem~\ref{them-coherent-regular-characterizations-flat} also holds for right  $R$-modules.
\end{enumerate}
\end{corollary}

%\begin{remark}\label{remark-emmanouil} For a general ring $R$, one may consider a stronger version of left regularity, where we demand that the projective dimensions of finitely presented left modules are uniformly bounded, and not just finite. Then any acyclic complex of projective \emph{left} $R$-modules is contractible. This follows by applying~\cite[Remark 2.11(i)]{emmanouil-tatelli-triviality-criteria} to the canonical projective cotorsion pair (Proj,All). If $R$ is left coherent, then this notion of strong left regularity is equivalent to  saying  $\text{w.dim}(R) < \infty$; see Example~\ref{example-regular-coherent-rings}(1). Over such a ring, by combining with Theorem~\ref{them-coherent-regular-characterizations-projective}, we see that acyclic complexes of projective left \emph{and}  right $R$-modules are contractible.%If $R$ is left (resp. right) coherent, then this notion of strong left (resp. right) regularity is equivalent to  saying  $\text{w.dim}(R) < \infty$; see Example~\ref{example-regular-coherent-rings}(1). Thus if $\text{w.dim}(R) < \infty$, and $R$ is either left \emph{or} right coherent, we get that acyclic complexes of projective \end{remark}

 \section*{Acknowledgements}
 The authors would like to thank Ioannis Emmanouil  for providing valuable feedback on the first draft of this paper.  We also thank him for sharing his related work with Talelli, \cite{emmanouil-tatelli-triviality-criteria}, and for the insightful discussions that followed regarding the relationships between the two works.

\end{document}